\documentclass[12pt]{iopart}

\usepackage{iopams}
\usepackage{amssymb}
\usepackage{amsthm, amsopn}
\usepackage{ifthen}
\usepackage{hyperref}
\usepackage{graphicx}
\usepackage{psfrag}

\newcommand{\eps}{\varepsilon}
\DeclareMathOperator*{\argmin}{arg\,min}
\DeclareMathOperator*{\argmax}{arg\,max}

\DeclareMathOperator{\proj}{proj}
\newcommand{\field}[1]{\mathbb{#1}}
\newcommand{\R}{\field{R}}
\newcommand{\N}{\field{N}}
\DeclareMathOperator{\Sym}{Sym}

\DeclareMathOperator{\E}{\mathbb{E}}

\newcommand{\inner}[3][n]{\SwitchBracketsizeLeft{#1}\LeftBracketSize\langle#2,#3\SwitchBracketsizeRight{#1}\RightBracketSize\rangle}
\newcommand{\abs}[2][n]{\SwitchBracketsizeLeft{#1}\LeftBracketSize|#2\SwitchBracketsizeRight{#1}\RightBracketSize|}
\newcommand{\norm}[2][n]{\SwitchBracketsizeLeft{#1}\LeftBracketSize\|#2\SwitchBracketsizeRight{#1}\RightBracketSize\|}
\newcommand{\set}[3][b]{\SwitchBracketsizeLeft{#1}\LeftBracketSize\{#2:#3\SwitchBracketsizeRight{#1}\RightBracketSize\}}

\newcommand{\NextScriptStyle}[1]{{\scriptstyle{#1}}}
\newcommand{\NextScriptScriptStyle}[1]{{\scriptscriptstyle{#1}}}
\newcommand{\NextTextStyle}[1]{{\textstyle{#1}}}
\newcommand{\NextDisplayStyle}[1]{{\displaystyle{#1}}}
\newcommand{\SwitchBracketsizeLeft}[1]{
  \ifthenelse{\equal{#1}{b}\OR\equal{#1}{big}}{\let\LeftBracketSize=\bigl}{
    \ifthenelse{\equal{#1}{B}\OR\equal{#1}{Big}}{\let\LeftBracketSize=\Bigl}{
      \ifthenelse{\equal{#1}{g}\OR\equal{#1}{bigg}}{\let\LeftBracketSize=\biggl}{
    \ifthenelse{\equal{#1}{G}\OR\equal{#1}{Bigg}}{\let\LeftBracketSize=\Biggl}{
      \ifthenelse{\equal{#1}{s}\OR\equal{#1}{small}}{\let\LeftBracketSize=\NextScriptStyle}{
        \ifthenelse{\equal{#1}{ss}}{\let\LeftBracketSize=\NextScriptScriptStyle}{
          \ifthenelse{\equal{#1}{t}\OR\equal{#1}{text}}{\let\LeftBracketSize=\NextTextStyle}{
        \ifthenelse{\equal{#1}{d}\OR\equal{#1}{display}}{\let\LeftBracketSize=\NextDisplayStyle}{
          \ifthenelse{\equal{#1}{a}\OR\equal{#1}{auto}}{\let\LeftBracketSize=\left}{
            \let\LeftBracketSize=\relax}}}}}}}}}}
\newcommand{\SwitchBracketsizeRight}[1]{
  \ifthenelse{\equal{#1}{b}\OR\equal{#1}{big}}{\let\RightBracketSize=\bigr}{
    \ifthenelse{\equal{#1}{B}\OR\equal{#1}{Big}}{\let\RightBracketSize=\Bigr}{
      \ifthenelse{\equal{#1}{g}\OR\equal{#1}{bigg}}{\let\RightBracketSize=\biggr}{
    \ifthenelse{\equal{#1}{G}\OR\equal{#1}{Bigg}}{\let\RightBracketSize=\Biggr}{
      \ifthenelse{\equal{#1}{s}\OR\equal{#1}{small}}{\let\RightBracketSize=\NextScriptStyle}{
        \ifthenelse{\equal{#1}{ss}}{\let\RightBracketSize=\NextScriptScriptStyle}{
          \ifthenelse{\equal{#1}{t}\OR\equal{#1}{text}}{\let\RightBracketSize=\NextTextStyle}{
        \ifthenelse{\equal{#1}{d}\OR\equal{#1}{display}}{\let\RightBracketSize=\NextDisplayStyle}{
          \ifthenelse{\equal{#1}{a}\OR\equal{#1}{auto}}{\let\RightBracketSize=\right}{
            \let\RightBracketSize=\relax}}}}}}}}}}

\theoremstyle{plain}

\newtheorem{theorem}{Theorem}[section]
\newtheorem{proposition}[theorem]{Proposition}
\newtheorem{corollary}[theorem]{Corollary}

\theoremstyle{definition}
\newtheorem{rem}[theorem]{Remark}
\newtheorem{assumption}[theorem]{Assumption}
\newtheorem{lemma}[theorem]{Lemma}
\newtheorem{definition}[theorem]{Definition}

\makeatletter
\newcommand{\logmessage}[1]{\@latex@warning{#1}}
\newcommand{\ignore}{\logmessage{Text ignored}\@gobble}
\makeatother

\begin{document}

\title{Nonparametric instrumental regression with non-convex constraints}

\author{M Grasmair${}^1$, O Scherzer${}^{1,2}$ and A Vanhems${}^{3}$}
\address{${}^1$Computational Science Center, University of Vienna, Austria}
\address{${}^2$Radon Institute of Computational and Applied Mathematics, Linz, Austria}
\address{${}^3$University of Toulouse, Toulouse Business School and Toulouse School of Economics, France}

\eads{\mailto{markus.grasmair@univie.ac.at}, \mailto{otmar.scherzer@univie.ac.at},
  \mailto{a.vanhems@esc-toulouse.fr}}

\begin{abstract}
This paper considers the nonparametric regression model with an additive error
that is dependent on the explanatory variables. As is common in empirical studies in
epidemiology and economics, it also supposes that valid instrumental variables are
observed. A classical example in microeconomics considers the consumer demand function
as a function of the price of goods and the income, both variables often considered as endogenous.
In this framework, the economic theory also imposes shape restrictions on the demand function,
like integrability conditions. Motivated by this illustration in microeconomics,
we study an estimator of a nonparametric constrained regression
function using instrumental variables by means of Tikhonov regularization.
We derive rates of convergence for the regularized model both in a deterministic
and stochastic setting under the assumption that the true regression function
satisfies a projected source condition including, because of the
non-convexity of the imposed constraints, an additional smallness condition.
\end{abstract}


\submitto{\IP}

\section{Motivation}

We consider the model
\[
Y_i = g(X_i) + \eps_i,\qquad i=1,\ldots,n,
\]
where $(Y_i,X_i)_{i=1,\ldots,n}$ is a sample of observations of size $n$ representing respectively the measured data
and variables effecting the measurements. The function $g$ describes the dependence of the data on the variables,
and $\eps_i$ is a combination of noise (measurement errors)
and \emph{modeling errors}, often resulting from the omittance of relevant variables. The goal is the estimation of the function $g$.
If the modeling errors $\eps$ and the variables $X$ are not dependent,
that is, if the conditional expectation $\E(\eps|X)$ of $\eps$ given $X$ is zero,
then it is possible to identify $g$ by
\begin{equation}\label{eq:g}
g(x) := \E(Y|X=x).
\end{equation}
If, however, the conditional expectation of $\eps$ given $X$ does not vanish,
then this will lead to a biased estimate, as
\[
\E(Y|X=x) = g(x) + \E(\eps|X=x).
\]
The variables $X$ are then called endogenous variables.
This issue of endogeneity typically arises in the presence of modeling errors,
in particular, if variables have been omitted from the model that
simultaneously influence both $X$ and $Y$. This has been illustrated in several applications,
for example in epidemiology (see \cite{Chen.et.al2008,DMS10,VBBG11}) and in economics
(see \cite{W08book} and also the survey~\cite{AngKru01}).
In the classical microeconomic setting of consumer demand, the endogeneity issue has also been raised.
In this framework, the variable $Y$ represents the observed demand of a consumer for $k$ goods,
and the explanatory variables $X$ include the vector of prices $P$ of the goods and the total budget $Z>0$ of the consumer;
the function $g \colon \R^{k}_{> 0} \times \R_{> 0} \to \R_{\ge 0}^k$ denotes the consumer demand.
The problem of price endogeneity has been highlighted in several research articles
(see for example \cite{BrownWalker89,Lewbel01,Matzkin07}).
In an industrial organization framework, the paper by \cite{BerryLEvinsohnPakes95} analyzes demand and supply
in differentiated product markets (like the US automobile industry) and highlight the problem involved
by correlation between prices and product characteristics, some of which are observed by the consumer but not by the econometrician.
Similarly total expenditure endogeneity has been studied in particular for Engel Curves analysis, see for example \cite{bck2007}.
\medskip

One remedy is the usage of \emph{instruments},
that is, different variables $W$, which influence both $P$ and $Z$
but are uncorrelated with $\eps$ (see~\cite{AngKru01} for an overview).
The analysis of nonparametric instrumental regression has been conducted in several
works such as \cite{DFFR11,F03eswc,HalHor05,Vanhems2010}.
Therefore we consider the model
\[
Y = g(X) + \eps
\]
and we assume that the random variable $X=(P,Z)$ is described by instruments $W$
in such a way that $\E(\eps|W) = 0$. Therefore, the equation~\eref{eq:g}
can be transformed into
\begin{equation}\label{eq:T1}
\E(g(X)|W=w)=\E(Y|W=w).
\end{equation}

We assume in the following that the relation between $Y$, $X$ and $W$
is described by a joint density $f_{YXW}\colon\Omega_Y\times\Omega_X\times\Omega_W\to \R_{> 0}$,
where, for simplicity, the finite measure spaces $\Omega_Y$, $\Omega_X$ and $\Omega_W$ are assumed
to be normalized. We consider $L^2$ spaces with respect to this joint probability density and
denote for example by $L^2(\Omega_X)$ functions depending on $P$ and $Z$ only.
In addition, we denote by $f_{YW}$, $f_{XW}$, $f_W$ the corresponding
marginal densities defined by
\begin{eqnarray*}
f_{YW}(y,w) &= \int_{\Omega_X} f_{YXW}(y,x,w)\,dx,\\
f_{XW}(x,w) &= \int_{\Omega_Y} f_{YXW}(y,x,w)\,dy,\\
f_{W}(w) &= \int_{\Omega_X}\int_{\Omega_Y} f_{YXW}(y,x,w)\,dx\,dy.
\end{eqnarray*}

Now assume that the set $\Omega$ is bounded and
$f_{YXW}$ is bounded away from zero.
We consider the operator $T\colon L^2(\Omega_X)\to L^2(\Omega_W)$ defined by
\begin{equation}\label{eq:T2}
T\psi(w)
:= \E(\psi(X)|W=w)
= \int_{\Omega_X} \psi(x)\frac{f_{XW}(x,w)}{f_W(w)}\,dx.
\end{equation}
Then \eref{eq:T1} can be rewritten as the Fredholm integral equation
\begin{equation}\label{eq:eq2}
Tg=h,
\end{equation}
where
\[
h(w)=\E(Y|W=w) = \int_{\Omega_Y} y\,\frac{f_{YW}(y,w)}{f_W(w)}\,dy.
\]

In addition, classical microeconomic theory imposes some shape restrictions on the consumer demand,
and the challenge is to take these constraints into account in the nonparametric estimation of the function $g$.
More precisely, standard micro-economic theory (see~\cite{Var92}) states that the demand is the result
of the maximization of some (unknown) utility function.
That is, there exists some function $u\colon \R^k_{\ge 0} \to \R$ (the utility) such that
\begin{equation}\label{eq:utility}
  g(x) = \argmax\set[b]{u(y)}{y \in \R^k_{\ge 0},\ \inner{y}{p} \le z},
\end{equation}
where $x=(p,z)$.
Here the utility function is assumed to be continuously differentiable,
concave, and strictly monotoneously increasing.
Even though the utility is unknown, the assumption of its existence
(and of utility maximization) has some implications for the
demand function $g$, called the \textit{integrability conditions}.
First, it is rather obvious that $g$ is homogeneous of degree 0,
that is, $g(tx) = g(x)$ for every $t > 0$.
Moreover, the maximum in~\eref{eq:utility} is always attained
at the boundary; more precisely, we have the equality
\begin{equation}\label{eq:budget}
\inner{y}{g(x)} = z;
\end{equation}
this condition is usually called the \emph{budget constraint}.
Finally, defining the \emph{Slutsky matrix}
\[
S_g(x) := \nabla_p g(x) + \partial_z g(x)\cdot g(x)^T,
\]
the conditions
\begin{equation}\label{eq:Slutsky}
S_g(x) = S_g(x)^T
\qquad\textrm{ and }\qquad
S_g(x) \le 0
\end{equation}
hold. That is, the Slutsky matrix is symmetric and negative semi-definite in (almost)
every point $x=(p,z)$.

Therefore, the objective of this work is to recover the function $g$ characterized by
equation \eref{eq:eq2} and satisfying the constraints defined by the Slutsky matrix.

The paper is organized as follows: In Section 2, we present our model, the link with ill-posed
inverse problems in the case where the transform is unknown, and the conditions under which a
regularized solution can be defined. In Section 3 we derive rates of convergence in a deterministic
setting and we extend the results in Section 4 to the statistical setting. 

\section{Constrained Inversion of $T$}

Let now $T$ be the operator defined in~\eref{eq:T2}
(operating on vector valued functions).
Then, in order to recover $g$, we have to solve the equation
\[
Tg = h,
\]
where $h$ denotes the right hand side of~\eref{eq:eq2}
subject to the constraints that $g$ is homogeneous of degree 0 and
satisfies the budget constraint~\eref{eq:budget} and the Slutsky condition~\eref{eq:Slutsky}
almost everywhere in $\Omega := \Omega_X = \Omega_P \times \Omega_Z$.
In the following we will always assume that the set
$\Omega$ is bounded, open, connected and has a Lipschitz boundary.

Apart from the constraints, there are three problems:
First, the operator $T$ is defined by the density $f_{XW}$,
which is not known exactly but can only by estimated up to a certain error $\delta$.
Consequently, we will only have an approximation $T^\delta$ of $T$ available.
Second, the right hand side $h$ is only known up to some error $\gamma$,
as it may be prone to measurement errors (in a deterministic setting)
or is the realization of a random variable (in a stochastic setting),
and, again, it depends on the density $f_{YW}$.
In addition, the assumption $\E(\eps|w) = 0$ need not hold exactly.
Finally, the operator $T$ (and also its approximation $T^\delta$)
is not boundedly invertible in $L^2(\Omega;\R^k)$.
Thus a direct solution of the operator equation
\[
T^\delta g = h^\gamma
\]
does not make sense, as its solution $g^{\delta,\gamma}$ (if it exists) need not be close
to the true solution $g^\dagger$, even if the errors $\delta$ and $\gamma$ are small.
In addition, there is no reason why the exact solution of the perturbed operator
equation (if it exists) should satisfy the required constraints,
in particular, as the constraints are non-linear and describe a non-convex set.

In order to find a solution nevertheless, it is necessary to consider
some kind of regularized solution.
In the following, we consider the application of (constrained) Tikhonov regularization,
where we use the (weighted) first order Sobolev norm as regularization functional.
That is, denoting for $\mu \ge 0$ by
\begin{equation}\label{eq:mu}
  \norm{g}_\mu^2 := \mu\norm{g}_{L^2}^2 + \norm{\nabla g}_{L^2}^2
\end{equation}
the weighted Sobolev norm,
one minimizes, for some regularization parameter $\alpha > 0$
depending on $\delta$ and $\gamma$, the functional
\[
\mathcal{T}_\alpha(g;T^\delta,h^\gamma) :=\norm{T^\delta g - h^\gamma}_{L^2}^2 + \alpha\norm{g}_\mu^2
\]
subject to the constraints of positivity, 0-homogeneity, the Slutsky condition,
and the budget constraint.
For the sake of simplicity, we will omit in the following
the subscripts in the $L^2$-norms and we will assume that
$\Omega$ is compactly contained in $\R_{>0}^k \times \R_{>0}$.

We use in the following the abbreviation
\[
\fl
\mathcal{X} := \set{g \in H^{1}(\Omega;\R^k)}{g \ge 0 \textrm{ is 0-homogeneous, }
  \inner{p}{g(x)} = z \textrm{ and } S_g = S_g^T \le 0\ \textrm{a.e.}}.
\]
Then one can define
\[
g_\alpha^{\delta,\gamma} := \argmin\set{\norm{T^\delta g-h^\gamma}^2 + \alpha\norm{g}_\mu^2}{g \in \mathcal{X}},
\]
provided the Tikhonov functional attains its minimum in $\mathcal{X}$.
In the following, we will show that this is indeed the case.
The proof is based on the direct method in the calculus of variations.
As a first important result, we prove that the set $\mathcal{X}$ is weakly closed in
$H^1(\Omega;\R^k)$, which is not an obvious assertion, as $\mathcal{X}$ is non-convex,
and the weak closedness of a subset of a Hilbert space is usually strongly
tied to its convexity.

\begin{lemma}
  The set $\mathcal{X}$ is weakly sequentially closed in $H^1(\Omega;\R^k)$.
\end{lemma}

\begin{proof}
  Obviously the set of non-negative $0$-homogeneous functions
  satisfying the budget constraint $\inner{p}{g(x)} = z$ is convex and closed in $H^1(\Omega;\R^k)$,
  implying that it is also weakly closed.

  Next we show that the mapping $S\colon H^1(\Omega;\R^k) \to L^1(\Omega;\R^{k\times k})$,
  \[
  g \mapsto S(g) = \nabla_p g + \partial_z g \cdot g^T
  \]
  is weak--weak continuous.
  To that end assume that the sequence $(g_n)_{n\in\N}$
  weakly converges to $g \in H^1(\Omega;\R^k)$.
  Then $\nabla g_n$ weakly converges to $\nabla g$ in $L^2(\Omega;\R^{k\times(k+1)})$
  (which in particular implies that the sequence is bounded)
  and the Rellich--Kondrachov compactness theorem (see~\cite[Thm.~6.2]{Ada75}) implies
  that the functions $g_n$ converge strongly to $g$
  with respect to the $L^2$ topology.
  Thus, if $1 \le i,j \le k$ and $u \in L^2(\Omega;\R)$, we have
  \begin{eqnarray*}
    \fl
  \abs{\inner{\partial_z g_n^{(i)} g_n^{(j)}-\partial_z g^{(i)}  g^{(j)}}{u}}
  &\le \abs{\inner{\partial_z g_n^{(i)} (g_n^{(j)} - g^{(j)})}{u}}
  + \abs{\inner{(\partial_z g_n^{(i)} - \partial_z g^{(i)}) g^{(j)}}{u}}\\
  &\le \norm{g_n^{(j)}-g^{(j)}}\norm{u}\norm{\partial_z g_n^{(i)}}
  + \abs{\inner{\partial_z g_n^{(i)} - \partial_z g^{(i)}}{g^{(j)}u}}
  \to 0.
  \end{eqnarray*}
  Consequently the product $\partial_z g_n \cdot g_n^T$
  converges to $\partial_z g \cdot g^T$ with respect
  to the weak topology on $L^1(\Omega;\R^{k\times k})$.

  Now note that the set $\Sym_k^-$ of all symmetric and negative semi-definite
  $(k\times k)$-matrices is a closed and convex cone in $\R^{k\times k}$.
  Consequently also the set of all summable functions
  on $\Omega$ with values in $\Sym_k^-$ is a closed and convex
  cone in $L^1(\Omega;\R^{k\times k})$ and therefore,
  in particular, also weakly closed.
  Therefore the weak-weak continuity of the mapping $S$ implies that
  the set of functions $g \in H^1(\Omega;\R^k)$ satisfying the Slutsky
  condition $S(g) = S(g)^T \le 0$ is weakly closed.

  This shows that the set $\mathcal{X}$ is the intersection
  of the (weakly closed) set of 0-homogeneous, non-negative functions satisfying
  the budget constraint with a weakly closed set,
  which proves that $\mathcal{X}$ itself is weakly closed in $H^1(\Omega;\R^k)$.
\end{proof}

For the usage of the direct method in the calculus of variations,
we still have to prove the coercivity of the regularization functional.
In the case $\mu > 0$, the coercivity is obvious, as the regularization term
is equivalent to the $H^1$-norm;
in the case $\mu = 0$, however, the equivalence only holds, if the
operator $T$ does not annihilate constant functions
(see~\cite{AubVes97,Ves01} for a related result on total variation regularization).
In the next result, we provide a detailed proof of this assertion
by explicitly computing constants defining this equivalence of norms.
In particular, the results show that these constants depend
continuously on the operator $T$, which will be required in the proof
of the convergence result, where we also treat the case of operator errors.

\begin{lemma}\label{le:equiv_norm}
  Assume that
  $T\colon L^2(\Omega;\R^k)\to L^2(\Omega;\R^k)$
  is a bounded linear operator.
  If $\mu = 0$, assume in addition that
  $Tc \neq 0$ for every non-zero constant function $c \colon \Omega\to \R^k$.
  Define for $g \in H^1(\Omega;\R^k)$
  \begin{equation}\label{eq:normT}
  \norm{g}_T^2 := \norm{g}^2_\mu + \norm{Tg}^2_{L^2}.
  \end{equation}
  Then $\norm{\cdot}_T$ is a norm on $H^1(\Omega;\R^k)$ that
  is equivalent to the standard $H^1$-norm.
  More precisely, we have the following estimates:
  For every $\mu \ge 0$,
  \begin{equation}\label{eq:equiv_norm1}
  \norm{g}_T \le \norm{T}\norm{g}_{H^1};
  \end{equation}
  if $\mu > 0$ in~\eref{eq:normT}, then
  \begin{equation}\label{eq:equiv_norm2}
  \norm{g}_{H^1} \le \frac{1}{\min\{\sqrt{\mu},1\}}\norm{g}_T,
  \end{equation}
  and if $\mu = 0$ in~\eref{eq:normT}, then there exists a constant $A > 0$
  only depending on the set $\Omega$ such that
  \begin{equation}\label{eq:equiv_norm3}
  \norm{g}_{H^1} \le A(\norm{T}D(T)^{-1}+D(T)^{-1}+1)\norm{g}_T,
  \end{equation}
  where
  \[
  D(T) := \inf\set{\norm{Tc}}{c\colon \Omega\to \R^k \textrm{ is constant with } \abs{c}=1}.
  \]
\end{lemma}

\begin{proof}
  Inequality~\eref{eq:equiv_norm1} follows from
  \[
  \norm{g}_T \le \norm{Tg}_{L^2} \le \norm{T}\norm{g}_{L^2} \le \norm{T}\norm{g}_{H^1},
  \]
  and~\eref{eq:equiv_norm2} is trivial.

  Now assume that $\mu = 0$.
  Then the assertion $Tc \neq 0$ for every non-zero constant function $c\colon \Omega\to \R^k$
  implies that $0 < D(T) < +\infty$.
  Define now the projection
  $P\colon L^2(\Omega;\R^k) \to L^2(\Omega;\R^k)$,
  $g\mapsto \frac{1}{\abs{\Omega}} \int_\Omega g$.
  Then
  \begin{eqnarray*}
    \norm{g}^2 &= \norm{g-Pg}^2 + \norm{Pg}^2\\
    &\le \norm{g-Pg}^2 + D(T)^{-2}\norm{TPg}^2\\
    &\le \norm{g-Pg}^2 + 2D(T)^{-2}\bigl(\norm{Tg}^2+\norm{T(g-Pg)}^2\bigr)\\
    &\le (1+2D(T)^{-2}\norm{T}^2)\norm{g-Pg}^2+2D(T)^{-2}\norm{Tg}^2.
  \end{eqnarray*}
  From the Poincar\'e Inequality (see e.g.~\cite[Thm.~4.8.1]{Zie89})
  it follows that there exists $C > 0$ such that
  $\norm{g-Pg}\le C\norm{\nabla g}$.
  Thus
  \begin{eqnarray*}
  \norm{g}_{H^1}^2 &= \norm{g}^2 + \norm{\nabla g}^2\\
  &\le \bigl(C^2(1+2D(T)^{-2}\norm{T}^2)+1\bigr)\norm{\nabla g}^2 + 2D(T)^{-2}\norm{Tg}^2.
  \end{eqnarray*}
  Setting $A = 2(C+1)$ we obtain~\eref{eq:equiv_norm3}.
\end{proof}

\begin{lemma}
  Assume that $T^\delta\colon L^2(\Omega;\R^k) \to L^2(\Omega;\R^k)$
  is bounded linear, $h^\gamma \in L^2(\Omega;\R^k)$, $\alpha > 0$,
  and $\mu \ge 0$.
  If $\mu = 0$, assume in addition that
  $T^\delta c \neq 0$ for every non-zero constant function $c\colon \Omega \to \R^k$.
  Then the regularization functional $\mathcal{T}_\alpha\colon L^2(\Omega;\R^k) \to \R_{\ge 0}\cup \{+\infty\}$.
  \[
  \mathcal{T}_\alpha(g;T^\delta,h^\gamma) :=
  \cases{\norm{T^\delta g - h^\gamma}^2+\alpha\norm{g}_\mu^2
    & if $g \in \mathcal{X}$,\\
    +\infty & else,}
  \]
  attains its minimum.
\end{lemma}

\begin{proof}
  The weak closedness of the set $\mathcal{X}$
  and the weak lower semi-continuity of the mapping
  $g \mapsto \frac{1}{2}\norm{T^\delta g-h^\gamma}^2 + \frac{\alpha}{2}\norm{g}_\mu^2$
  on the space $H^1(\Omega)$ imply that also the mapping
  $\mathcal{T}_\alpha(\cdot;T^\delta,h^\gamma)$ is weakly lower semi-continuous.
  Moreover, Lemma~\ref{le:equiv_norm} implies that
  $\mathcal{T}_\alpha(\cdot;T^\delta,h^\gamma)$ is weakly coercive.
  Applying the direct method in the calculus of variations,
  we obtain the existence of a minimizer.
\end{proof}

Note that the previous result does not say
anything about the uniqueness of the minimizer.
Because of the non-convexity of the set $\mathcal{X}$,
it is probable that the Tikhonov functional
has multiple local minima, but also possible
that it has several global minima.

The following result is very similar to the convergence result
in~\cite{NeuSch90}. The main difference is that we also
consider the homogeneous Sobolev semi-norm as a regularization term,
which is not coercive by itself.
The coercivity (or rather the equi-coercivity of the functionals
$\mathcal{T}_\alpha(\cdot;T^\delta,h^\gamma)/\alpha$)
is only obtained by means of Lemma~\ref{le:equiv_norm}.

\begin{proposition}\label{pr:wellposed}
  Assume that $T\colon L^2(\Omega;\R^k) \to L^2(\Omega;\R^k)$ is bounded
  linear satisfying $Tc \neq 0$ for every non-zero constant function $c\colon\Omega\to\R^k$
  and that the operator equation $Tg = h$ has a solution in $\mathcal{X}$.
  Let $\delta_j \to 0$, $\gamma_j \to 0$ and assume that
  $T^{\delta_j}\colon L^2(\Omega;\R^k) \to L^2(\Omega;\R^k)$ are bounded linear operators
  satisfying $\norm{T^{\delta_j}-T} \le \delta_j$
  and that the functions $h^{\gamma_j} \in L^2(\Omega;\R^k)$ satisfy $\norm{h^{\gamma_j}-h} \le \gamma_j$.
  Let $\mu \ge 0$ be fixed; if $\mu = 0$, assume in addition that
  $T^{\delta_j}c \neq 0$ for every non-zero constant function $c\colon\Omega\to\R^k$.

  Assume that $\alpha_j > 0$ is chosen such that $\alpha_j \to 0$
  and $(\delta_j+\gamma_j)^2/\alpha_j \to 0$.
  Then every sequence $(g_j)_{j\in\N} \subset \mathcal{X}$ satisfying
  \[
  g_j \in \argmin\set{\mathcal{T}_{\alpha_j}(g;T^{\delta_j},h^{\gamma_j})}{g \in \mathcal{X}}
  \]
  has a subsequence $g_{j^{(i)}}$ converging with respect
  to the $H^1$-norm to some
  \[
  g^\dagger \in \argmin\set{\norm{g}_\mu^2}{Tg = h,\ g\in \mathcal{X}}.
  \]
\end{proposition}

\begin{proof}
  Let $\tilde{g}$ be any solution of $Tg=h$ in $\mathcal{X}$.
  Then
  \begin{eqnarray*}
    \norm{T^{\delta_j}g_j-h^{\gamma_j}}^2 + \alpha\norm{g_j}_\mu^2
    &\le \norm{T^{\delta_j}\tilde{g}-h^{\gamma_j}}^2+\alpha\norm{\tilde{g}}_\mu^2\\
    &\le \bigl(\norm{T^{\delta_j}-T}\norm{\tilde{g}}+\norm{h-h^{\gamma_j}}\bigr)^2 + \alpha\norm{\tilde{g}}_\mu^2\\
    &\le (\delta_j\norm{\tilde{g}}+\gamma_j)^2 + \alpha\norm{\tilde{g}}_\mu^2.
  \end{eqnarray*}
  Consequently,
  \begin{eqnarray*}
  \norm{g_j}_{T^{\delta_j}}^2
  &\le 2\norm{Tg_j-h^{\gamma_j}}^2 + 2\norm{h^{\gamma_j}}^2
  + \norm{g_j}_\mu^2\\
  &\le (\delta_j\norm{\tilde{g}}+\gamma_j)^2 + \alpha\norm{\tilde{g}}_\mu^2 + 2\norm{h^{\gamma_j}}^2
  + \frac{(\delta_j\norm{\tilde{g}}^2+\gamma_j)^2}{\alpha} + \norm{\tilde{g}}_\mu^2.
  \end{eqnarray*}
  From Lemma~\ref{le:equiv_norm},
  it follows that $\norm{g_j}_{T^{\delta_j}} \ge C(T^{\delta_j})\norm{g_j}_{H^1}$
  for some constants $C(\cdot) > 0$ depending continuously
  on the operator $T^{\delta_j}$.
  Therefore, the assumption $(\delta_j+\gamma_j)^2/\alpha \to 0$
  implies that the sequence $(g_j)_{j\in\N}$ is bounded.
  The proof of the subsequential convergence is now
  along the lines of~\cite[Thm.~3.26]{SchGraGroHalLen09}.
\end{proof}

\section{Convergence Rates}

\begin{lemma}\label{le:var_ineq}
  Assume that $T\colon L^2(\Omega;\R^k) \to L^2(\Omega;\R^k)$ is bounded
  linear and that the equation $Tg=h$ has a solution in $\mathcal{X}$.
  Let
  \[
  g^\dagger \in \argmin\set{\norm{g}_\mu^2}{Tg = h,\ g\in \mathcal{X}}.
  \]
  Let moreover $T^\delta\colon L^2(\Omega;\R^k) \to L^2(\Omega;\R^k)$
  satisfy $\norm{T^\delta-T}_{L^2} \le \delta$,
  and let $h^\gamma \in L^2(\Omega;\R^k)$ satisfy $\norm{h^\gamma-h} \le \gamma$.
  If $\mu = 0$, assume in addition that $\delta < \norm{T}$ and
  $Tc \neq 0$ for every non-zero constant function $c\colon\Omega\to\R^k$.
  Assume that there exists a set $L \subset \mathcal{X}$ such that,
  for some $\beta > 0$, $C \ge 0$ and every $g \in L$, we have
  \begin{equation}\label{eq:varineq}
    \beta\norm{g-g^\dagger}_\mu^2 \le \norm{g}_\mu^2
    - \norm{g^\dagger}_\mu^2 + C\norm{T(g-g^\dagger)}.
  \end{equation}
  Let moreover
  \[
  g_\alpha^{\delta,\gamma} \in \argmin\set{\mathcal{T}_\alpha(g;T^\delta,h^\gamma)}{g \in \mathcal{X}}.
  \]
  Define for $\mu > 0$
  \[
  D_\mu(\alpha,\delta,\gamma) := \frac{\delta\norm{g^\dagger}+\gamma}{\sqrt{\mu\alpha}} + \frac{\norm{g^\dagger}_\mu}{\sqrt{\mu}},
  \]
  and let
  \[
  D_0(\alpha,\delta,\gamma) := A\frac{\norm{T}+D(T)+1}{D(T)-\delta}
  \biggl[\norm{h}+\gamma + \sqrt{2}\frac{\delta\norm{g^\dagger}+\gamma + \sqrt{\alpha}\norm{\nabla g^\dagger}}{\min\{\sqrt{\alpha},1\}}\biggr]
  \]
  with $A > 0$ and $D(T) > 0$ as in Lemma~\ref{le:equiv_norm}.

  Then the estimates
  \[
  \beta\norm{g_\alpha^{\delta,\gamma} - g^\dagger}_\mu^2 \le \frac{(\gamma+\delta\norm{g^\dagger})^2}{\alpha}
  + C\bigl(\gamma+\delta D_\mu(\alpha,\delta,\gamma)\bigr) + \frac{C^2\alpha}{4}
  \]
  and
  \[
  \norm{T(g_\alpha^{\delta,\gamma}-g^\dagger)}^2 \le 2(\gamma+\delta\norm{g^\dagger})^2
  + 2\alpha C\bigl(\gamma+\delta D_\mu(\alpha,\delta,\gamma)\bigr) + C^2\alpha^2
  \]
  hold whenever $g_\alpha^{\delta,\gamma} \in L$.
\end{lemma}

\begin{proof}
  The inequality~\eref{eq:varineq} and the optimality of $g_\alpha^{\delta,\gamma}$ imply that
  \begin{eqnarray*}
    \beta\norm{g_\alpha^{\delta,\gamma}-g^\dagger}_\mu^2
    & \le \norm{g_\alpha^{\delta,\gamma}}_\mu^2 - \norm{g^\dagger}_\mu^2
    + C\norm{T(g_\alpha^{\delta,\gamma}-g^\dagger)}\\
    & \le \frac{1}{\alpha}\Bigl(\norm{T^\delta g^\dagger-h^\gamma}^2-\norm{T^\delta g_\alpha^{\delta,\gamma}-h^\gamma}^2\Bigr)
    + C\norm{T(g_\alpha^{\delta,\gamma}-g^\dagger)}\\
    & \le \frac{(\delta\norm{g^\dagger}_{L^2}+\gamma)^2}{\alpha}+C\bigl(\delta\norm{g_\alpha^{\delta,\gamma}}_{L^2}+\gamma\bigr) \\
    &\qquad\qquad + C\norm{T^\delta g_\alpha^{\delta,\gamma}-h^\gamma}
    - \frac{\norm{T^\delta g_\alpha^{\delta,\gamma}-h^\gamma}^2}{\alpha}.
  \end{eqnarray*}
  Estimating
  \[
  C\norm{T^\delta g_\alpha^{\delta,\gamma}-h^\gamma} - \frac{\norm{T^\delta g_\alpha^{\delta,\gamma}-h^\gamma}^2}{\alpha}
  \le \sup_{t\ge 0} \Bigl[Ct-\frac{t^2}{\alpha}\Bigr] = \frac{C^2\alpha}{4},
  \]
  we obtain the inequality
  \begin{equation}\label{eq:rate_1}
    \beta\norm{g_\alpha^{\delta,\gamma} - g^\dagger}_\mu^2 \le \frac{(\gamma+\delta\norm{g^\dagger}_{L^2})^2}{\alpha}
    + C\bigl(\gamma+\delta\norm{g_\alpha^{\delta,\gamma}}\bigr) + \frac{C^2\alpha}{4}.
  \end{equation}
  Moreover, using the estimate
  \begin{eqnarray*}
    C\norm{T^\delta g_\alpha^{\delta,\gamma}-h^\gamma} - \frac{\norm{T^\delta g_\alpha^{\delta,\gamma}-h^\gamma}^2}{\alpha}\\
    \le \sup_{t\ge 0} \Bigl[Ct-\frac{t^2}{2\alpha}\Bigr] - \frac{\norm{T^\delta g_\alpha^{\delta,\gamma}-h^\gamma}^2}{2\alpha}
    = \frac{C^2\alpha}{2}-\frac{\norm{T^\delta g_\alpha^{\delta,\gamma}-h^\gamma}^2}{2\alpha},
  \end{eqnarray*}
  we obtain
  \begin{equation}\label{eq:rate_2}
    \norm{T(g_\alpha^{\delta,\gamma}-g^\dagger)}^2 \le 2(\gamma+\delta\norm{g^\dagger}_{L^2})^2
    + 2\alpha C\bigl(\gamma+\delta\norm{g_\alpha^{\delta,\gamma}}\bigr) + C^2\alpha^2.
  \end{equation}

  Assume first that $\mu > 0$.
  Then the definition of $\norm{\cdot}_\mu$
  and the optimality of $g_\alpha^{\delta}$ imply the estimate
  \[
  \fl
  \norm{g_\alpha^{\delta,\gamma}} \le \frac{1}{\sqrt{\mu}}\norm{g_\alpha^{\delta,\gamma}}_\mu
  \le \frac{1}{\sqrt{\mu}}\Bigl(\frac{(\delta\norm{g^\dagger}+\gamma)^2}{\alpha} + \norm{g^\dagger}_\mu^2\Bigr)^{1/2}
  \le \frac{1}{\sqrt{\mu}}\Bigl(\frac{\delta\norm{g^\dagger}+\gamma}{\sqrt{\alpha}} + \norm{g^\dagger}_\mu\Bigr),
  \]
  which proves the assertion for $\mu$ strictly positive.
  \medskip

  Now assume that $\mu = 0$.
  Then Lemma~\ref{le:equiv_norm} implies that,
  using the same notation as in the Lemma,
  \[
  \norm{g_\alpha^{\delta,\gamma}} \le
  A(\norm{T^\delta}D(T^\delta)^{-1}+D(T^\delta)^{-1}+1)\norm{g_\alpha^{\delta,\gamma}}_{T^\delta}.
  \]
  Moreover, for $\norm{T} > \delta$, we have
  \begin{eqnarray*}
  D(T^\delta)
  &= \inf\set{\norm{T^\delta c}}{c\colon\Omega\to\R^k \textrm{ is constant with }\abs{c}=1}\\
  &\ge \inf\set{\norm{Tc}}{c\colon\Omega\to\R^k \textrm{ is constant with }\abs{c}=1}-\delta\\
  &= D(T)-\delta.
  \end{eqnarray*}
  Therefore,
  \begin{eqnarray*}
    \norm{g_\alpha^{\delta,\gamma}} \le
    A((\norm{T}+\delta)(D(T)-\delta)^{-1}+(D(T)-\delta)^{-1}+1)\norm{g_\alpha^{\delta,\gamma}}_{T^\delta}\\
    = A\frac{\norm{T}+D(T)+1}{D(T)-\delta}\norm{g_\alpha^{\delta,\gamma}}_{T^\delta}.
  \end{eqnarray*}
  Now the optimality of $g_\alpha^{\delta,\gamma}$ implies that
  \begin{eqnarray*}
    \norm{g_\alpha^{\delta,\gamma}}_{T^\delta}
    &\le \norm{T^\delta g_\alpha^{\delta,\gamma}-h^\gamma}+\norm{\nabla g_\alpha^{\delta,\gamma}} + \norm{h^\gamma}\\
    &\le \norm{h}+\gamma + \frac{\bigl(2\mathcal{T}_\alpha(g_\alpha^{\delta,\gamma};T^\delta,h^\gamma)\bigr)^{1/2}}{\min\{\sqrt{\alpha},1\}}\\
    &\le \norm{h}+\gamma + \frac{\bigl(2\mathcal{T}_\alpha(g^\dagger;T^\delta,h^\gamma)\bigr)^{1/2}}{\min\{\sqrt{\alpha},1\}}\\
    &\le \norm{h}+\gamma + \sqrt{2}\frac{\delta\norm{g^\dagger}+\gamma + \sqrt{\alpha}\norm{\nabla g^\dagger}}{\min\{\sqrt{\alpha},1\}}.
  \end{eqnarray*}
  Together, these estimates show that
  \[
  \norm{g_\alpha^{\delta,\gamma}}
  \le A\frac{\norm{T}+D(T)+1}{D(T)-\delta}
  \biggl[\norm{h}+\gamma + \sqrt{2}\frac{\delta\norm{g^\dagger}+\gamma
      + \sqrt{\alpha}\norm{\nabla g^\dagger}}{\min\{\sqrt{\alpha},1\}}\biggr].
  \]
  Inserting this inequality in~\eref{eq:rate_1} and~\eref{eq:rate_2}
  proves the assertion for $\mu = 0$.
\end{proof}

In the next result, we will present concrete conditions
that imply the inequality~\eref{eq:varineq}.
These conditions are a generalization of \emph{projected source conditions},
which are a classical concept in the theory of inverse problems
with convex contraints (see~\cite{ChaKun94a,Eic92,Neu87}), to a non-convex setting.
Recently, the relation between projected source conditions
and variational inequalities of the type~\eref{eq:varineq}
has also been studied in~\cite{FleHof11}, though still in a convex setting.
In order to generalize this concept to non-convex constraints,
we recall the notion of a proximal normal cone to
a subset of a Hilbert space (see~\cite{ClaLedSteWol98}).

\begin{definition}
  Let $Y$ be a Hilbert space and let $S \subset Y$ be non-empty.
  We define for $y \in Y$ the set $\proj_S(y) \subset S$
  as the set of all points $z \in S$ for which the distance
  to $y$ is minimal.
  Moreover we define for $z \in S$ the \emph{proximal normal cone} $N_S^P(z)$ to $S$
  at $z$ as
  \[
  N_S^P(z) := \set{\zeta = t(y-z) \in Y}{t\ge 0,\ z \in \proj_S(y)}.
  \]
  See also Figure~\ref{fig:prox_normal_cone}.
\end{definition}

\begin{figure}
  \psfrag{X}{$S$}
  \psfrag{NX}{$N_S^P(z)$}
  \psfrag{g}{$z$}
  \centering
  \includegraphics[width=0.4\textwidth]{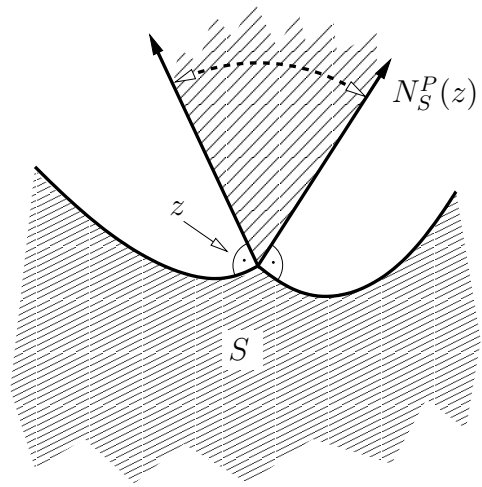}
  \caption{Proximal normal cone to the non-convex set $S$ at the point $z \in \partial S$.}
  \label{fig:prox_normal_cone}
\end{figure}

For the following result, see~\cite[Prop.~1.5]{ClaLedSteWol98}.

\begin{proposition}
  A vector $\zeta$ belongs to $N_S^P(z)$, if and only if there
  exists $\tau \ge 0$ (possibly depending on $\zeta$ and $z$) such that
  \begin{equation}\label{eq:npchar}
    \inner{\zeta}{y-z} \le \tau \norm{y-z}^2
  \end{equation}
  for all $y \in S$.
\end{proposition}

In the following we will denote, for given $z \in S$
and $\zeta \in N_S^P(z)$, by $\tau(\zeta,z)$
the smallest $\tau \ge 0$ for which~\eref{eq:npchar} holds.
Then the function $\tau$ is positively homogeneous
with respect to its first variable, that is,
$\tau(t\zeta,z) = t\tau(\zeta,z)$ whenever $\zeta\in N_S^P(z)$
and $t > 0$ (note that the fact that $N_S^P(z)$ is a cone
implies that $t\zeta \in N_S^P(z)$).

\begin{theorem}\label{th:pr_source}
  Assume that $g^\dagger \in \mathcal{X}$ satisfies $Tg^\dagger = h$.
  In addition, assume that $\partial_\nu g^\dagger = 0$ on $\partial \Omega$.
  Denote moreover by $T^*\colon L^2(\Omega;\R^k) \to L^2(\Omega;\R^k)$
  the adjoint of $T$
  and let $N_{\mathcal{X}}^P(g^\dagger) \subset L^2(\Omega;\R^{k})$
  be the proximal normal cone to the set $\mathcal{X}$ at the point $g^\dagger$.

  Assume that there exist $\omega\in L^2(\Omega;\R^k)$ and $\zeta \in N_{\mathcal{X}}^P(g^\dagger)$
  such that
  \[
  2(\mu g^\dagger - \Delta g^\dagger) = T^*\omega + \zeta.
  \]

  \begin{itemize}
  \item If $\mu > 0$ and $\tau(\zeta,g^\dagger) < \mu$, then~\eref{eq:varineq}
    holds for every $g \in \mathcal{X}$ with $C = \norm{\omega}$ and $\beta = 1-\tau(\zeta,g^\dagger)/\mu$.
  \item If $\mu = 0$, assume in addition that $Tc \neq 0$ for every
    non-zero constant function $c\colon\Omega\to\R^k$ and that
    \[
    E :=  A^2(\norm{T}D(T)^{-1}+D(T)^{-1}+1)^2\,\tau(\zeta,g^\dagger)
    \]
    with $A$ and $D(T)$ as in Lemma~\ref{le:equiv_norm}
    satisfies $E < 1$.
    Then for every $s > 0$ the inequality~\eref{eq:varineq}
    holds with $\beta = 1-E$ and $C = \norm{\omega}+sE$
    whenever $g \in \mathcal{X}$ satisfies $\norm{T(g-g^\dagger)} < s$.
  \end{itemize}
\end{theorem}

\begin{proof}
  First note that
  \begin{eqnarray*}
  \inner{2\mu g^\dagger - 2\Delta g^\dagger-\zeta}{g^\dagger-g}
  &= \inner{T^*\omega}{g^\dagger-g}\\
  &= \inner{\omega}{T(g^\dagger-g)}\\
  &\le \norm{\omega}\norm{T(g^\dagger-g)}.
  \end{eqnarray*}
  Now the assumption $\zeta \in N_{\mathcal{X}}^P(g^\dagger)$ implies that
  \[
  \inner{\zeta}{g^\dagger-g} \le \tau(\zeta,g^\dagger)\,\norm{g^\dagger-g}^2
  \]
  for all $g \in \mathcal{X}$.
  In addition, Stoke's theorem and the assumption $\partial_\nu g^\dagger = 0$
  on $\partial\Omega$ imply that
  \begin{eqnarray*}
    2\inner{\mu g^\dagger - \Delta g^\dagger}{g^\dagger-g}
    = 2\mu\inner{g^\dagger}{g^\dagger-g} + 2\inner{\nabla g^\dagger}{\nabla(g^\dagger-g)}\\
    = \norm{g^\dagger-g}_\mu^2 + \norm{g^\dagger}_\mu^2 - \norm{g}_\mu^2.
  \end{eqnarray*}
  Thus we obtain the estimate
  \begin{equation}\label{eq:est1}
    \norm{\omega}\norm{T(g-g^\dagger)}
    \ge \norm{g^\dagger}_\mu^2 - \norm{g}_\mu^2 + \norm{g^\dagger-g}_\mu^2 - \tau(\zeta,g^\dagger)\,\norm{g^\dagger-g}^2.
  \end{equation}
  In the case $\mu > 0$, it follows that
  \[
  (1-\tau(\zeta,g^\dagger)/\mu) \norm{g^\dagger-g}_\mu^2 \le \norm{g^\dagger}_\mu^2 - \norm{g}_\mu^2 + \norm{\omega}\norm{T(g-g^\dagger)},
  \]
  which proves the first part of the assertion.

  On the other hand, if $\mu = 0$, then~\eref{eq:est1} and Lemma~\ref{le:equiv_norm} imply that
  \[
  \fl
  (1-E)\norm{\nabla(g^\dagger-g)}^2
  \le \norm{\nabla g}^2-\norm{\nabla g^\dagger}^2 + \norm{\omega}\norm{T(g-g^\dagger)}
  + E\norm{T(g-g^\dagger)}^2.
  \]
  Thus~\eref{eq:varineq} holds for $\norm{T(g-g^\dagger)} \le s$.
\end{proof}

\begin{corollary}\label{co:rate}
  Assume that the assumptions of Theorem~\ref{th:pr_source} are satisfied.
  Then we have, with the notation of Lemma~\ref{le:var_ineq}, the estimates
  \[
  \fl
  \bigl(1-\tau(\zeta,g^\dagger)/\mu\bigr) \norm{g_\alpha^{\delta,\gamma}-g^\dagger}_\mu^2
  \le \frac{(\gamma+\delta\norm{g^\dagger})^2}{\alpha}
  + \norm{\omega}\bigl(\gamma+\delta D_\mu(\alpha,\delta,\gamma)\bigr) + \frac{\norm{\omega}^2\alpha}{4}
  \]
  in the case $\mu > 0$, and
  \begin{eqnarray*}
  \fl
  (1-E)\norm{\nabla(g_\alpha^{\delta,\gamma}-g^\dagger)}^2\\
  \le \frac{(\gamma+\delta\norm{g^\dagger})^2}{\alpha}
  + (\norm{\omega}+sE)\bigl(\gamma+\delta D_0(\alpha,\delta,\gamma)\bigr) + \frac{(\norm{\omega}+sE)^2\alpha}{4}
  \end{eqnarray*}
  in the case $\mu = 0$.
  In particular, we have in both cases with a parameter choice
  $\alpha \asymp \max\{\delta,\gamma\}$ a convergence rate
  \[
  \norm{g_\alpha^{\delta,\gamma}-g^\dagger}_\mu^2 = O(\max\{\delta,\gamma\}).
  \]
\end{corollary}

\begin{rem}
  Consider for the moment the setting where the constraint set $\mathcal{X}$
  is closed and convex.
  Then the convexity of $\mathcal{X}$ implies that
  $\tau(\zeta,g^\dagger) = 0$ whenever $\zeta\in N_{\mathcal{X}}^P(g^\dagger)$;
  in other words, the proximal normal cone $N_{\mathcal{X}}^P(g^\dagger)$
  coincides with the (usual) normal cone
  $N_{\mathcal{X}}(g^\dagger) = \set{\zeta}{\inner{\zeta}{\tilde{g}-g} \le 0 \textrm{ for all } \tilde{g}\in \mathcal{X}}$.
  Thus in the condition $T^*\omega + \zeta = 2(\mu g^\dagger-\Delta g^\dagger)$
  for some $\zeta \in N_{\mathcal{X}}^P(g^\dagger)$ no smallness condition is required
  for $\zeta$, and therefore this condition reduces to the classical projected
  source condition found in~\cite{ChaKun94a,Neu87}.
\end{rem}

\begin{rem}
  The conditions and results of Theorem~\ref{th:pr_source}
  and Corollary~\ref{co:rate} can also be translated into the context
  of convex analysis with subgradients and Bregman distances
  (see~\cite{BurOsh04,HofKalPoeSch07,SchGraGroHalLen09}).
  Recall that the subdifferential $\partial\mathcal{R}(g^\dagger) \subset X$
  of a convex mapping $\mathcal{R}\colon X \to [0,+\infty]$ at $g^\dagger$
  consists of all elements $\xi \in X$ satisfying
  $\mathcal{R}(g) \ge \mathcal{R}(g^\dagger) + \inner{\xi}{g-g^\dagger}$
  for all $g \in X$.
  Moreover, the Bregman distance $\mathcal{D}^\xi(\cdot;g^\dagger)$
  is defined as
  \[
  \mathcal{D}^\xi(g;g^\dagger)
  := \mathcal{R}(g)-\mathcal{R}(g^\dagger)-\inner{\xi}{g-g^\dagger}.
  \]

  If $\mathcal{R}(g) := \norm{g}_\mu^2$
  (setting $\mathcal{R}(g) = +\infty$ if $g \not\in H^1(\Omega;\R^k)$),
  we obtain that the subdifferential is non-empty if and only if
  $\partial_\nu g^\dagger = 0$ on $\partial\Omega$.
  Moreover, in this case its unique element is
  the function $2(\mu g^\dagger - \Delta g^\dagger)$.
  Finally, it is easy to see that the Bregman distance
  between with $g$ and $g^\dagger$ with respect to $\norm{\cdot}_\mu^2$
  is precisely $\norm{g-g^\dagger}_\mu^2$.

  In this setting, Corollary~\ref{co:rate} with $\mu > 0$ reads as follows:
  If there exist $\xi \in \partial\mathcal{R}(g^\dagger)$ and
  $\zeta \in N_{\mathcal{X}}^P(g^\dagger)$ with $\tau(\zeta,g^\dagger) < \mu$,
  then
  \[
  \mathcal{D}^\xi(g_\alpha^{\delta,\gamma}) = O(\max\{\delta,\gamma\}).
  \]

  Note moreover that in~\cite{Gra10} a theory based on abstract convex analysis
  has been developed in order to derive convergence rates for non-convex
  regularization terms.
  Again, the results of Corollary~\ref{co:rate} can be seen as special cases
  of the results in~\cite[Section~4]{Gra10}
  by realizing that the function $2(\mu g^\dagger-\Delta g^\dagger) - \zeta$
  is a generalized subgradient of the mapping
  \[
  \mathcal{R}(g) = \cases{
    \norm{g}_\mu^2 & if $g \in \mathcal{X}$,\\
    + \infty & else.
  }
  \]
\end{rem}

\section{Extension to the stochastic setting}

In this section, we allow the approximation errors $\norm{T^{\delta}-T}$ and $\norm{h^{\gamma}-h}$
to be stochastic and depend on the sample size $n$.
More precisely, $T^{\delta}$ is a nonparametric estimator of the operator $T$
depending on the random sample $(Y_i,X_i,W_i)_{i=1,\ldots,n}$ and we will denote it by
$\hat{T}$. Similarly, $h^{\gamma}$ is a nonparametric estimator of the function $h$
depending on the random sample $(Y_i,X_i,W_i)_{i=1,\ldots,n}$ and we will denote it by $\hat{h}$ .
Finally, the approximated regularized solution $g_\alpha^{\delta,\gamma}$ will be denoted by $\hat{g}_{\alpha}$.

In the following, we will derive convergence rates in probability
for $\hat{g}_\alpha$.
To that end,
recall that a sequence of random variables $Q_n$, $n \in \N$,
in a normed space is bounded in probability, if for every $\epsilon > 0$ there exists
$C > 0$ and $n_0 \in \N$ such that
\[
\mathbb{P}(\norm{Q_n} > C) < \epsilon
\qquad
\textrm{ for all } n \ge n_0.
\]
In this case, we say that
\[
Q_n = O_P(1).
\]
Similarly, if $c_n$, $n \in \N$, is any real sequence,
we write
\[
Q_n = O_P(c_n)
\qquad\textrm{ if }\qquad
\frac{Q_n}{c_n} = O_P(1).
\]

Note that an alternative to convergence rates in probability
is the derivation of convergence rates in expectation,
which has been carried out for Tikhonov regularization
and generalizations in~\cite{BisHohMun04,BisHohMunRuy07}.
In this paper, however, we will restrict ourselves to rates in probability
in order to be able to exploit the results in~\cite{DFFR11} on
unconstrained instrumental regression.

Following \cite{DFFR11}, we introduce the kernel approach with generalized kernel functions of order $l$
for estimating $\hat{T}$ and $\hat{h}$. Note that the kernel is considered in generalized
form only to overcome edge effects. Let $\sigma \equiv \sigma_n \rightarrow 0$ denote a bandwidth and
$K_{\sigma}(\cdot,\cdot)$ denote a univariate generalized kernel function with the properties
$K_{\sigma}(u,t)=0$ if $u>t$ or $u<t-1$; for all $t\in[0,1]$,
\begin{equation*}
\sigma^{-(j+1)}\int_{t-1}^{t}u^{j}K_{{\sigma}}(u,t)du=
\cases{
1& if $j = 0$,\\
0& if $1 \le j \le l-1$.
}
\end{equation*}
We call $K_{\sigma}(\cdot,\cdot)$ a univariate generalized kernel function of order $l$
(see \cite{Muller1991}). A special class of multivariate generalized kernel functions
of order $l$ is given by that of products of univariate generalized kernel functions of order $l$.
Let $K_{X,\sigma}$ and $K_{W,\sigma}$ denote two generalized multivariate kernel functions of dimension $k+1$
and $K_{Y,\sigma}$ a kernel function of dimension $1$.
First we estimate the density functions $f_{YW}$, $f_{XW}$ and $f_{W}$.
Note that, for simplicity of notation, we use the same bandwidth to estimate the three densities
\begin{eqnarray*}
\hat{f}_{YW}(y,w) &= \frac{1}{n\sigma^{k+2}}\sum_{i=1}^nK_{Y,\sigma}(y-Y_i,y)K_{W,\sigma}(w-W_i,w),\\
\hat{f}_{XW}(x,w) &= \frac{1}{n\sigma^{2k+2}}\sum_{i=1}^nK_{X,\sigma}(x-X_i,x)K_{W,\sigma}(w-W_i,w),\\
\hat{f}_{W}(w) &= \frac{1}{n\sigma^{k+1}}\sum_{i=1}^nK_{W,\sigma}(w-W_i,w).
\end{eqnarray*}

Then the estimators of $T$ and $h$ are
\begin{eqnarray*}
\hat{T}\psi(w)&=\int\psi(x)\frac{\hat{f}_{XW}(x,w)}{\hat{f}_W(w)}dx, \\
\hat{h}(w)&=\int y \frac{\hat{f}_{YW}(y,w)}{\hat{f}_W(w)}dy.
\end{eqnarray*}

In order to derive a rate of convergence for $\hat{g}_{\alpha}$, we
require

\begin{assumption}\label{ass:stoch}
  We assume that the following conditions are satisfied:
  \begin{enumerate}
  \item The data $(Y_i,X_i,W_i)$, $i=1,\ldots,n$, define an \emph{i.i.d.}\ sample of $(Y,X,W)$.
  \item The probability density function $f_{YXW}$ is $l$ times continuously
    differentiable in the interior of $\Omega_Y \times \Omega \times \Omega_W$
    and bounded away from zero on $\Omega_Y \times \Omega \times \Omega_W$.
  \item The conditional expectation $\E(\epsilon^2 |W=w)$ is uniformly bounded on $\Omega_W$.
  \item Both multivariate kernels $K_{X,\sigma}$ and $K_{W,\sigma}$ are product kernels
    generated from the univariate generalized kernel function $K_{\sigma}$
    with the following properties:
    \begin{enumerate}
    \item The kernel function $K_{\sigma}$ is a generalized kernel function of order $l$.
    \item For each $t \in [0,1]$, the function $K_{\sigma}(\sigma\cdot,t)$ is supported on a
      set of the form $[(t-1)/\sigma,t/\sigma]\cap \mathcal{K}$ where $\mathcal{K}$ is a compact interval
      not depending on $t$ and $\sup_{\sigma>0,t\in [0,1],u\in \mathcal{K}}\abs{K_{\sigma}(\sigma u,t)} < \infty$.
    \end{enumerate}
  \item The bandwidth parameter satisfies $\sigma \rightarrow 0$ and $(n\sigma^{2k+2})^{-1}\log(n) \rightarrow 0$.
  \end{enumerate}
\end{assumption}

\begin{proposition}
  Suppose Assumption~\ref{ass:stoch} holds. Let $\rho=\min\{l,k+1\} \geq 2$
  and $\mu \ge 0$.
  Let
  \[
  g^\dagger \in \argmin\set{\norm{g}_\mu^2}{Tg = h,\ g\in \mathcal{X}}.
  \]
  and
  \[
  \hat{g}_\alpha \in \argmin\set{\mathcal{T}_\alpha(g;\hat{T},\hat{h})}{g \in \mathcal{X}}.
  \]

  Assume that $\partial_\nu g^\dagger = 0$ on $\partial \Omega$.
  Denote moreover by $T^*\colon L^2(\Omega;\R^k) \to L^2(\Omega;\R^k)$
  the adjoint of $T$
  and let $N_{\mathcal{X}}^P(g^\dagger) \subset L^2(\Omega;\R^{k+1})$
  be the proximal normal cone to the set $\mathcal{X}$ at the point $g^\dagger$.

  \begin{enumerate}
  \item Let $\mu > 0$.
    Assume that there exist $\omega\in L^2(\Omega;\R^k)$ and $\zeta \in N_{\mathcal{X}}^P(g^\dagger)$
    with $\tau(\zeta,g^\dagger) < \mu$
    such that
    \[
    2(\mu g^\dagger - \Delta g^\dagger) = T^*\omega + \zeta.
    \]
    Then the estimate
    \[
    \norm{\hat{g}_\alpha - g^\dagger}_\mu^2 = O_P\left((1+\sqrt{\alpha}) \frac{\frac{1}{n \sigma^{2k+2}}+\sigma^{2\rho}}{\alpha}
    + \frac{1}{n \sigma^{k+1}}+\sigma^{\rho} + \alpha\right)
    \]
    holds.
  \item Let $\mu = 0$ and assume that $\alpha \to 0$, $\Bigl(\frac{1}{n\sigma^{2k+2}}+\sigma^{2\rho}\Bigr)/\alpha \to 0$
    as $n \to \infty$.
    Assume moreover that there exist $\omega\in L^2(\Omega;\R^k)$ and $\zeta \in N_{\mathcal{X}}^P(g^\dagger)$
    such that
    \[
    2(\mu g^\dagger - \Delta g^\dagger) = T^*\omega + \zeta
    \]
    and
    \[
    A^2(\norm{T}D(T)^{-1}+D(T)^{-1}+1)^2\,\tau(\zeta,g^\dagger) < 1,
    \]
    where $A$ and $D(T)$ are as in Lemma~\ref{le:equiv_norm}.
    Then the estimate
    \[
    \norm{\nabla\hat{g}_\alpha - \nabla g^\dagger}^2 = O_P\left( \frac{\frac{1}{n \sigma^{2k+2}}+\sigma^{2\rho}}{\alpha}
    + \frac{1}{n \sigma^{k+1}}+\sigma^{\rho} + \alpha\right)
    \]
    holds.
  \end{enumerate}

  In particular, if
  \[
  \alpha \asymp n^{-\frac{\rho}{2(k+\rho+1)}}
  \qquad\textrm{ and }\qquad
  \sigma \asymp n^{-\frac{1}{2(k+\rho+1)}},
  \]
  then we obtain in both cases the rate
  \[
  \norm{\hat{g}_\alpha - g^\dagger}_\mu^2 = O_P\bigl(n^{-\frac{\rho}{2(k+\rho+1)}}\bigr).
  \]
\end{proposition}

\begin{proof}
  Note first that the assumption that the density
  $f_{YXW}$ is bounded away from zero implies that the
  operator $T$ is bounded and satisfies $Tc \neq 0$
  for every constant function $c$.
  Moreover, in~\cite{DFFR11} the convergence rate result
  \begin{eqnarray*}
    \norm{\hat{T}-T}^2 &=& O_P\left(\frac{1}{n \sigma^{2k+2}}+\sigma^{2\rho}\right),\\
    \norm{\hat{h}-h}^2 &=& O_P\left(\frac{1}{n \sigma^{2k+2}}+\sigma^{2\rho}\right)
  \end{eqnarray*}
  has been derived under Assumption~\ref{ass:stoch}.
  Together with the results of Lemma~\ref{le:var_ineq}, Theorem~\ref{th:pr_source}
  and Corollary~\ref{co:rate}, this immediately proves the assertion in
  the case $\mu > 0$.

  In the case $\mu = 0$, note that the assumption on the behaviour
  of $\alpha$ and Proposition~\ref{pr:wellposed} imply
  that the regularized solutions $\hat{g}_\alpha$ converge in probability to $g^\dagger$.
  Moreover, the convergence in probability of $\hat{T}$ to $T$
  implies that $1/(D(T)-\norm{\hat{T}-T}) = O_P(1)$, and therefore,
  as $\sigma \to 0$ and $1/(n\sigma^{2k+2}) \to 0$,
  we obtain in the notation of Lemma~\ref{le:var_ineq} the estimate
  \[
  D_0(\alpha,\norm{\hat{T}-T},\norm{\hat{h}-h}) = O_P(1).
  \]
  Then the result follows again immediately from Lemma~\ref{le:var_ineq}, Theorem~\ref{th:pr_source}
  and Corollary~\ref{co:rate}.
\end{proof}

\section{Conclusion}

In this paper, we have studied the problem of nonparametric regression
in the presence of endogenous variables and
additional non-convex shape constraints.
The main motivation is the estimation of the consumer demand function,
which, according to standard microeconomic theory, satisfies certain
(non-linear) integrability conditions.
We have used instruments in order to tackle the issue of endogeneity,
which, in the case where the coupling between the instruments
and the explanatory variables is weak (that is, only given by a density),
leads to the solution of an ill-posed operator equation.

We propose to solve the resulting inverse problem
by (constrained) Tikhonov regularization using a weighted Sobolev norm
as a regularization term.
Because of the weak closedness of the constrained set in the Sobolev
space, the regularization method is convergent.
In addition, we have derived convergence rates under the additional
assumption that the true solution $g^\dagger$ satisfies a certain variational
inequality, which is shown to hold if $g^\dagger$ satisfies a projected
source condition. In contrast to the usual convex case, however,
this condition is coupled with a smallness condition.
The convergence rates are derived in both a deterministic and
a stochastic setting. In the latter situation we have the additional
problem that the correspondence between the instruments and the
explanatory variables, and thus the operator itself,
is not known exactly but has to be estimated in a first step.
Here we propose to use a kernel estimator, which allows us
to obtain rates in probability for the operator error in dependence
of the number of measurements.


\section*{References}

\begin{thebibliography}{20}

\bibitem{Ada75}
R.~A. Adams.
\newblock {\em {S}obolev Spaces}.
\newblock Academic Press, New York, 1975.

\bibitem{AngKru01}
J.D. Angrist and A.B. Krueger.
\newblock Instrumental variables and the search for identifications: from
  supply and demand to natural experiments.
\newblock {\em J. Econ. Persp.}, 15(4):69--85, 2001.

\bibitem{AubVes97}
G.~Aubert and L.~Vese.
\newblock A variational method in image recovery.
\newblock {\em SIAM J. Numer. Anal.}, 34(5):1948--1979, 1997.

\bibitem{BerryLEvinsohnPakes95}
S.~Berry, J.~Levinsohn, and A.~Pakes.
\newblock Automobile prices in market equilibrium.
\newblock {\em Econometrica} 63:841–-890, 1995.

\bibitem{BisHohMun04}
N.~Bissantz, T.~Hohage, and A.~Munk.
\newblock Consistency and rates of convergence of nonlinear {T}ikhonov
  regularization with random noise.
\newblock {\em Inverse Probl.}, 20(6):1773--1789, 2004.

\bibitem{BisHohMunRuy07}
N.~Bissantz, T.~Hohage, A.~Munk, and F.~Ruymgaart.
\newblock Convergence rates of general regularization methods for statistical
  inverse problems and applications.
\newblock {\em SIAM J. Numer. Anal.}, 45(6):2610--2636 (electronic), 2007.

\bibitem{bck2007}
R.~Blundell, X.~Chen, and D.~Kristensen.
\newblock Nonparametric IV estimation of shape-invariant Engel curves
\newblock {\em Econometrica}, 75:1613--1669, 2007.

\bibitem{BrownWalker89}
B.~Brown, and M.~Walker.
\newblock The random utility hypothesis and inference in demand systems.
\newblock {\em Econometrica}, 57:815–29, 1989.

\bibitem{BurOsh04}
M.~Burger and S.~Osher.
\newblock Convergence rates of convex variational regularization.
\newblock {\em Inverse Probl.}, 20(5):1411--1421, 2004.

\bibitem{ChaKun94a}
G.~Chavent and K.~Kunisch.
\newblock Convergence of {T}ikhonov regularization for constrained ill-posed
  inverse problems.
\newblock {\em Inverse Probl.}, 10:63--76, 1994.

\bibitem{Chen.et.al2008}
L.~Chen, G.D.~Smith, R.~Harbord, and S.J.~Lewis.
\newblock Alcohol Intake and Blood Pressure: A Systematic Review Implementing a Mendelian Randomization Approach.
\newblock {\em PLoS Med}, 5(3), 2008.	

\bibitem{ClaLedSteWol98}
F.~H. Clarke, Yu.~S. Ledyaev, R.~J. Stern, and P.~R. Wolenski.
\newblock {\em Nonsmooth analysis and control theory}, volume 178 of {\em
  Graduate Texts in Mathematics}.
\newblock Springer-Verlag, New York, 1998.

\bibitem{DFFR11}
S.~Darolles, Y.~Fan, J.P.~Florens, and E.~Renault.
\newblock Nonparametric Instrumental Regression.
\newblock {\em Econometrica}, 79:1541--1565, 2011.

\bibitem{DMS10}
V.~Didelez, S.~Meng, and N.A.~Sheehan.
\newblock Assumptions of {IV} Methods for Observational Epidemiology.
\newblock {\em Statistical Science}, 25:22--40, 2010.

\bibitem{Eic92}
B.~Eicke.
\newblock Iteration methods for convexly constrained ill-posed problems in
  {H}ilbert space.
\newblock {\em Numer. Funct. Anal. Optim.}, 13(5-6):413--429, 1992.

\bibitem{FleHof11}
J.~Flemming and B.~Hofmann.
\newblock Convergence rates in constrained {T}ikhonov regularization:
  equivalence of projected source conditions and variational inequalities.
\newblock {\em Inverse Probl.}, 27(8):085001, 11, 2011.

\bibitem{F03eswc}
J.P.~Florens.
\newblock {\em Inverse Problems and Structural Econometrics: The Example of Instrumental Variables},
in Advances in Economics and Econometrics: Theory and Applications, 284--311.
\newblock Cambridge: Cambridge University Press, 2003.

\bibitem{Gra10}
M.~Grasmair.
\newblock Generalized {B}regman distances and convergence rates for non-convex
  regularization methods.
\newblock {\em Inverse Probl.}, 26(11):115014, 2010.

\bibitem{HalHor05}
P.~Hall and J.L.~Horowitz.
\newblock Nonparametric methods for inference in the presence of instrumental
  variables.
\newblock {\em Ann. Statist.}, 33(6):2904--2929, 2005.

\bibitem{HofKalPoeSch07}
B.~Hofmann, B.~Kaltenbacher, C.~P{\"o}schl, and O.~Scherzer.
\newblock A convergence rates result for {T}ikhonov regularization in {B}anach
  spaces with non-smooth operators.
\newblock {\em Inverse Probl.}, 23(3):987--1010, 2007.

\bibitem{Lewbel01}
A.~Lewbel.
\newblock Demand systems with and without errors.
\newblock {\em American Economic Review} 91:611–18, 2001.

\bibitem{Matzkin07}
R.~Matzkin.
\newblock {\em Heterogeneous choice}. In N. W. Blundell, R. and T. Persson (Eds.), Advances in
Economics and Econometrics, Theory and Applications: Ninth World Congress of the Econometrics
Society, Volume 43:111--121.
\newblock Cambridge: Cambridge University Press, 2007.

\bibitem{Muller1991}
H.-G. M\"uller.
\newblock Smooth Optimum Kernel Estimators Near Endpoints.
\newblock {\em Biometrika}, 78:521-530, 1991.

\bibitem{Neu87}
A.~Neubauer.
\newblock Finite-dimensional approximation of constrained
  {T}ikhonov-regularized solutions of ill-posed linear operator equations.
\newblock {\em Math. Comp.}, 48(178):565--583, 1987.

\bibitem{NeuSch90}
A.~Neubauer and O.~Scherzer.
\newblock Finite-dimensional approximation of {T}ikhonov regularized solutions
  of nonlinear ill-posed problems.
\newblock {\em Numer. Funct. Anal. Optim.}, 11(1-2):85--99, 1990.

\bibitem{SchGraGroHalLen09}
O.~Scherzer, M.~Grasmair, H.~Grossauer, M.~Haltmeier, and F.~Lenzen.
\newblock {\em Variational methods in imaging}, volume 167 of {\em Applied
  Mathematical Sciences}.
\newblock Springer, New York, 2009.

\bibitem{Vanhems2010}
A.~Vanhems.
\newblock Nonparametric estimation of exact consumer surplus with endogeneity in price.
\newblock {\em Econometrics Journal}, 13(3):80-98, 2010.

\bibitem{VBBG11}
S.~Vansteelandt, J.~Bowden, M.~Babanezhad, and E.~Goetghebeur.
\newblock On Instrumental Variables Estimation of Causal Odds Ratios.
\newblock {\em Statistical Science}, 26:403--422, 2011.

\bibitem{Var92}
H.~R. Varian.
\newblock {\em Microeconomic Analysis}.
\newblock W.~W.~Norton \& Company, New York, third edition, 1992.

\bibitem{Ves01}
L.~Vese.
\newblock A study in the {BV} space of a denoising-deblurring variational
  problem.
\newblock {\em Appl. Math. Optim.}, 44(2):131--161, 2001.

\bibitem{W08book}
J.~Wooldridge.
\newblock {\em Introductory Econometrics: A Modern Approach}.
\newblock South-Western College Pub, fourth edition, 2008.	

\bibitem{Zie89}
W.~P. Ziemer.
\newblock {\em Weakly Differentiable Functions. {S}obolev Spaces and Functions
  of Bounded Variation}, volume 120 of {\em Graduate Texts in Mathematics}.
\newblock Springer Verlag, Berlin etc., 1989.

\end{thebibliography}

\end{document}